\documentclass[11pt,a4paper]{article}
\usepackage{amsmath}
\usepackage{amsthm}
\usepackage{amsfonts}
\usepackage{eucal}

\parskip\smallskipamount
\newtheorem{theorem}{Theorem}
\newtheorem{lemma}[theorem]{Lemma}
\newtheorem{corollary}[theorem]{Corollary}
\theoremstyle{remark}
\newtheorem{remark}{Remark}

\renewcommand{\Pr}{{\bf P}}

\newcommand{\R}{{\mathbb R}}
\newcommand{\Rp}{{\mathbb R}^+}

\newcommand{\LL}{{\mathcal L}}
\newcommand{\SSR}{{\mathcal S}}
\newcommand{\ov}[1]{\overline{#1}}

\newcommand{\con}[4]{{#1}_{#2}*{#3}_{#4}}
\newcommand{\ms}[1]{m_{#1}}

\begin{document}
\begin{center}
  \Large\textbf{Convolutions of long-tailed and subexponential
    distributions} 
\end{center}

 \begin{center}
   Sergey Foss, Dmitry Korshunov and Stan Zachary
 \end{center}

 \begin{center}
   \textit{Heriot-Watt University and Sobolev Institute of Mathematics}
 \end{center}

 \begin{quotation}\small\noindent
   Convolutions of long-tailed and subexponential distributions play a
   major role in the analysis of many stochastic systems. We study
   these convolutions, proving some important new results through a
   simple and coherent approach, and showing also that the standard
   properties of such convolutions follow as easy consequences.



  \vskip 0.2cm
  \noindent
  {\it Short title.\/} Convolutions of distributions.

\end{quotation}
\section{Introduction}
\label{sec:introduction}

Heavy-tailed distributions play a major role in the analysis of many
stochastic systems.  For example, they are frequently necessary to
accurately model inputs to computer and communications networks, they
are an essential component of the description of many risk processes,
and they occur naturally in models of epidemiological spread.

Since the inputs to such systems are frequently cumulative in their
effects, the analysis of the corresponding models typically features
convolutions of such heavy-tailed distributions.  The properties of
such convolutions depend on their satisfying certain regularity
conditions.  From the point of view of applications practically all
such distributions may be considered to be long-tailed, and indeed to
possess the stronger property of subexponentiality (see below for
definitions).

In this paper we study convolutions of long-tailed and subexponential
distributions (probability measures), and (in passing) more general
finite measures, on the real line.  Our aim is to prove some important
new results, and to do so through a simple, coherent and systematic
approach.  It turns out that all the standard properties of such
convolutions are then obtained as easy consequences of these results.
Thus we also hope to provide further insight into these properties,
and to dispel some of the mystery which still seems to surround the
phenomenon of subexponentiality in particular.

Our approach is based on a simple decomposition for such convolutions,
and on the concept of ``$h$-insensitivity'' for a long-tailed
distribution or measure with respect to some (slowly) increasing
function $h$.  This novel approach and the basic, and very simple, new
results we require are given in Section~\ref{sec:basic-results}.  In
Section~\ref{sec:results-conv-long} we study convolutions of
long-tailed distributions.  The key results here are
Theorems~\ref{long.add.5} and \ref{long.add.5.plus}
which give conditions under which a random
shifting preserves tail equivalence; in the remainder of this section
we show how other (mostly known) results follow quickly and easily
from our approach, and provide some generalisations.  In
Section~\ref{sec:results-conv-subexp} we similarly study convolutions
of subexponential distributions.  The main results
here---Theorems~\ref{thm:s1} and \ref{thm:s2}---are new, as is
Corollary~\ref{cor:14}; again some classical results are immediate
consequences.  Finally, in Section~\ref{sec:clos-prop-subexp} we
consider closure properties for the class of subexponential
distributions.  Theorem~\ref{closure.of.S} gives a new necessary and
sufficient condition for the convolution of two subexponential
distributions to be subexponential, together with a simple
demonstration of the equivalence of some existing conditions.

Occasionally we take a few lines to reprove something from the
literature.  This enables us to give a self-contained treatment of our
subject.

Good introductions to the current state of knowledge on long-tailed
and subexponential distributions may be found in
Asmussen \cite{APQ,A}, Embrechts et al. \cite{EKM},
and Rolski et al. \cite{RSST}.

\section{Basic results}
\label{sec:basic-results}


We are concerned primarily with probability distributions.  However,
we find it convenient to work also with more general finite measures,
allowing these to be added, convoluted, etc, as usual.  As we
shall discuss further, our later results for distributions translate
easily into this more general setting, where they then provide
additional insight.

Recall that, for any two measures $F$ and $G$ and for any two non-negative
numbers $p$ and $q$, their {\it mixture} $pF+qG$ is also a finite
measure defined by $(pF+qG)(B) = pF(B)+qG(B)$, for any Borel set $B$.

Recall also that a (finite) measure $F$ on $\R$ is \emph{long-tailed}
if and only if $\ov F(x)>0$ for all $x$ and, for all constants $a$,
\begin{equation}\label{eq:1}
  \ov F(x+a)=\ov F(x)+o(\ov F(x))
  \quad\mbox{ as }x\to\infty,
\end{equation}
where the \emph{tail function}~$\ov F$ of the measure~$F$ is
given by $\ov F(x)=F(x,\infty)$.  It is easy to see that it is
sufficient that the relation~\eqref{eq:1} hold for some $a\ne0$.

We note also that the class of long-tailed measures has the following
readily verified closure property, which we henceforth use without
comment: if measures $F_1$, \dots, $F_n$ are long-tailed and if the
measure $F$ is such that $\ov F(x)\sim\sum_{k=1}^nc_k\ov{F_k}(x)$
as $x\to\infty$, for some $c_1$, \dots, $c_k>0$, then $F$ is also
long-tailed.  (Here
and throughout we use ``$\sim$'' to mean that the ratio of the
quantities on either side of this symbol converges to one; thus, for
example, the relation \eqref{eq:1} may be written as $\ov
F(x+a)\sim\ov F(x)$ as $x\to\infty$; we further frequently omit,
especially in proofs, the qualifier ``as $x\to\infty$'', as all our
limits will be of this form.)

We shall make frequent use of the following construction: given any
long-tailed measure~$F$, we may choose a positive function~$h$ on
$\Rp=[0,\infty)$ such that $h(x)\to\infty$ as $x\to\infty$
and additionally $h$ is increasing
sufficiently slowly that
\begin{equation}
  \label{eq:2}
  \ov F(x\pm h(x)) \sim \ov F(x)
  \quad\mbox{ as }x\to\infty.
\end{equation}
For example, we may choose a sequence $x_n$ increasing to infinity
such that, for all $n$,
\begin{displaymath}
  |\ov F(x\pm n)-\ov F(x)| \le \ov F(x)/n
  \quad\mbox{ for all }x>x_n,
\end{displaymath}
and then set $h(x)=n$ for $x\in(x_n,x_{n+1}]$.  (The introduction of
such a function $h$ will allow us to avoid the general messiness of
repeatedly taking limits first as $x$ tends to infinity and then as a
further constant $a$---essentially that featuring in
\eqref{eq:1}---tends to infinity.)  Given a long-tailed measure~$F$ and
a function $h$ satisfying 
the above conditions, we shall
say that $F$ is $h$-\emph{insensitive}.

Note further that, given a finite collection of long-tailed measures
$F_1$, \dots, $F_n$, we may choose a function $h$ on $\Rp$ such that
each $F_i$ is $h$-insensitive.  For example, for each $i$ we may
choose $h_i$ such that $F_i$ is $h_i$-insensitive, and then define $h$
by $h(x)=\displaystyle\min_i h_i(x)$.

The tail function of the convolution of any two measures~$F$ and $G$
is given by
\begin{equation}\label{eq:8}
  \ov{\con{F}{}{G}{}}(x) = \int_{-\infty}^\infty \ov F(x-y)G(dy)
  = \int_{-\infty}^\infty \ov G(x-y)F(dy).
\end{equation}

For any measure $F$ and for any Borel set $B$ we denote by $F_B$ the
measure given by the restriction of $F$ to $B$, that is
$F_B(A)=F(A\cap B)$ for all Borel sets $A$.  We also use the
shortenings $F_{\le h}=F_{(-\infty,h]}$ and $F_{>h}=F_{(h,\infty)}$.
Finally, we define $\ms{F}=F(\R)$ to be the total mass associated with
the measure $F$.

Now let $h$ be any positive function on $\Rp$.
Then the tail function of the convolution of any two measures $F$
and $G$ possesses the following decomposition:
for $x\ge0$,
\begin{align}
  \ov{\con{F}{}{G}{}}(x)
  & = \ov{\con{F}{\le h}{G}{}}(x)
  + \ov{\con{F}{>h}{G}{}}(x)\label{eq:03},
\end{align}
and upper estimate:
\begin{align}
  \ov{\con{F}{}{G}{}}(x)
  & \le \ov{\con{F}{\le h}{G}{}}(x)
  + \ov{\con{F}{}{G}{\le h}}(x)
  + \ov{\con{F}{>h}{G}{>h}}(x),\label{eq:3.upper}
\end{align}
where in~\eqref{eq:03} and \eqref{eq:3.upper} above $h$ stands for
$h(x)$, so that, for example,
$F_{\le{}h}(B)=F_{\le{}h(x)}(B)=F(B\cap(-\infty,h(x)])$, again for any
Borel set~$B$. If in addition $h(x)\le x/2$, then
\begin{align}
  \ov{\con{F}{}{G}{}}(x)
  & = \ov{\con{F}{\le h}{G}{}}(x)
  + \ov{\con{F}{}{G}{\le h}}(x)
  + \ov{\con{F}{>h}{G}{>h}}(x),\label{eq:3}
\end{align}
because in this case
$\ov{\con{F}{\le h}{G}{}}(x)=\ov{\con{F}{\le h}{G}{>h}}(x)$
and $\ov{\con{F}{}{G}{\le h}}(x)=\ov{\con{F}{>h}{G}{\le h}}(x)$.
Note that
\begin{align}\label{eq:4}
  \ov{\con{F}{\le h}{G}{}}(x)
  & = \int_{-\infty}^{h(x)} \ov G(x-y)F(dy),\\
  \label{eq:4.1}
  \ov{\con{F}{>h}{G}{}}(x)
  & = \int_{-\infty}^\infty \ov
  F(\max(h(x),x-y))G(dy),
\end{align}
while $\ov{\con{F}{>h}{G}{>h}}$ is symmetric in $F$ and $G$ and
\begin{equation}
  \label{eq:5}
  \ov{\con{F}{>h}{G}{>h}}(x)
  = \int_{h(x)}^\infty \ov F(\max(h(x),x-y)) G(dy)
  = \int_{h(x)}^\infty \ov G(\max(h(x),x-y)) F(dy).
\end{equation}

Note also that if, on some probability space with probability
measure~$\Pr$, $\xi$ and $\eta$ are independent random variables with
respective distributions~$F$ and $G$ (with $\ms{F}=\ms{G}=1$), then
for $x\ge0$,
\begin{align*}
 \ov{\con{F}{\le h}{G}{}}(x) & = \Pr(\xi+\eta>x,\, \xi\le h(x)),\\
 \ov{\con{F}{>h}{G}{>h}}(x)
 & = \Pr(\xi+\eta > x,\, \xi > h(x), \,\eta > h(x)).
\end{align*}

The following four lemmas are the keys to everything that follows.
\begin{lemma}\label{lem:h1}
  Suppose that the measure $G$ is long-tailed and that $h$ is such
  that $G$ is $h$-insensitive.  Then, for any measure $F$,
  \begin{displaymath}
    \ov{\con{F}{\le h}{G}{}}(x)\sim \ms{F}\ov G(x)
    \qquad\mbox{as }x\to\infty.
  \end{displaymath}
\end{lemma}
\begin{proof}
  It follows from \eqref{eq:4} that
  $\ov{\con{F}{\le h}{G}{}}(x)\le \ms{F}\ov G(x-h(x))$.
  On the other hand,
    \begin{align*}
    \ov{\con{F}{\le h}{G}{}}(x)
    & \ge \ov{\con{F}{[-h,h]}{G}{}}(x)\nonumber\\
    & \ge F[-h(x),h(x)]\ov G(x+h(x)) \nonumber\\
    & \sim \ms{F}\ov G(x+h(x))
    \quad\mbox{ as }x\to\infty, \nonumber
  \end{align*}
  where the last equivalence follows since
  $h(x)\to\infty$ as $x\to\infty$. The required result now follows
  from the $h$-insensitivity of $G$.
\end{proof}

We now prove a version of Lemma~\ref{lem:h1} which is symmetric
in $F$ and $G$ and
will allow us to get many important
results for convolutions---see the further discussion below.

\begin{lemma}\label{lem:h2}
  Suppose that measures $F$ and $G$ are such that
  $\ms{F}G+\ms{G}F$ is long-tailed and that $h$ is such
  that $\ms{F}G+\ms{G}F$ is $h$-insensitive.
  Then, as $x\to\infty$,
  $$
  \ov{\con{F}{\le h}{G}{}}(x)
  +\ov{\con{F}{}{G}{\le h}}(x)
  \sim \ms{F}\ov G(x)+\ms{G}\ov F(x).
  $$
\end{lemma}
\begin{proof}
  It follows from \eqref{eq:4} that
  $$
  \ov{\con{F}{\le h}{G}{}}(x)+\ov{\con{F}{}{G}{\le h}}(x)
  \le \ms{F}\ov G(x-h(x))+\ms{G}\ov F(x-h(x)).
  $$
  On the other hand,
    \begin{align*}
    \ov{\con{F}{\le h}{G}{}}(x)
    +\ov{\con{F}{}{G}{\le h}}(x)
    & \ge \ov{\con{F}{[-h,h]}{G}{}}(x)
    +\ov{\con{F}{}{G}{[-h,h]}}(x)\nonumber\\
    & \ge F[-h(x),h(x)]\ov G(x+h(x))
    +G[-h(x),h(x)]\ov F(x+h(x)) \nonumber\\
    & \sim \ms{F}\ov G(x+h(x))+\ms{G}\ov F(x+h(x))
    \quad\mbox{ as }x\to\infty, \nonumber
  \end{align*}
  where the last equivalence follows since
  $h(x)\to\infty$ as $x\to\infty$. The required result now follows
  from the $h$-insensitivity of $\ms{F}G+\ms{G}F$.
\end{proof}

Note that special cases under which $\ms{F}G+\ms{G}F$ is long-tailed
are (a) $F$ and $G$ are both long-tailed, and (b) $F$ is long-tailed
and $\ov G(x)=o(\ov F(x))$ as $x\to\infty$.\footnote{ There are also
  other possibilities. For instance, here is an example where $F$ and
  $G$ are probability distributions such that $F$ is long-tailed and
  $F+G$ is long-tailed while $G$ does not satisfy (a) or (b).  Take
  $\alpha >0$ and let $\ov F(x) = x^{-\alpha}$ for $x\ge 1$ and
  $\ov F(x)=1$ for $x\le 1$. Let $\ov G(x) = 1$ for $x\le 1$ and
  construct
  $\ov G$ on $(1,\infty)$ inductively on the intervals
  $[x_n,x_{n+1}]$, with $x_1=1$. For $n=1,2,\ldots$,
  let $y_n = \min \{x>x_n \ : \ 1+ \ln (x/x_n) = 2^n\}$
  and then, for $x\in [x_n,y_n)$,
  let $\ov G(x) = \ov F(x)/(1+\ln(x/x_n))$.
  Further, let $\ov G(y_n) = \ov F(y_n) \cdot 2^{-n-1}$
  (this means that
  $\ov G(y_n)= \ov G(y_n-)/2$). Finally, let
  $x_{n+1} = \min \{ x>y_n \colon \ov F(x) =
  \ov G(y_n) \}$ and then
  let $\ov G(x) = \ov
  G(y_n)$ for $x\in [y_n,x_{n+1}]$.\\
  As $n$ increases the jumps of $\ov G$ at the points $y_n$
  become negligible with respect to $\ov F$, so $F+G$ is
  long-tailed. On the other hand, these jumps are not negligible
  with respect to $\ov G$ itself, and $G$ cannot be long-tailed. Also,
  $\ov G(x_n) = \ov F(x_n)$, and so condition (a) is violated.}
The importance of this
lemma arises because in many applications, while we may naturally be
concerned primarily with long-tailed measures (e.g.\ service
times in queueing models), we nevertheless require to make small
corrections arising from other input measures whose tails are
relatively lighter (e.g.\ those of interarrival times);
Lemma~\ref{lem:h2} is then necessary in order to obtain asymptotic
results such as Theorem~\ref{thm:s1} below, ensuring that such lighter
tails indeed make a negligible contribution.  In the case where we are
solely concerned with long-tailed measures, there are obvious
simplifications to the proofs of our main results.

Finally, the following  two simple lemmas will be useful
when we come to consider subexponential measures.

\begin{lemma}\label{lem:h3}
  Let $h$ be any positive function on $\Rp$ such that $h(x)\to\infty$.
  Then, for any measures $F_1$, $F_2$ and $G$ on $\R$,
  $$
    \limsup_{x\to\infty}
    \frac{\ov{\con{(F_1)}{>h}{G}{}}(x)}
    {\ov{\con{(F_2)}{>h}{G}{}}(x)}  \le
    \limsup_{x\to\infty}
    \frac{\ov{F_1}(x)}{\ov{F_2}(x)}.
  $$
  In particular, in the case where the limit of the ratio
  $\ov{F_1}(x)/\ov{F_2}(x)$ exists, we have
  \begin{displaymath}
    \lim_{x\to\infty}
    \frac{\ov{\con{(F_1)}{>h}{G}{}}(x)}
    {\ov{\con{(F_2)}{>h}{G}{}}(x)}
    =\lim_{x\to\infty}\frac{\ov{F_1}(x)}{\ov{F_2}(x)}.
  \end{displaymath}
\end{lemma}
\begin{proof}
  The results are immediate from (\ref{eq:4.1})
  and from the first of the integral
  representations in~\eqref{eq:5} on noting that
  $\max(h(x),x-y)\to\infty$
  as $x\to\infty$ uniformly in all $y\in\R$.
\end{proof}

Taking into account the symmetry of $\ov{\con{F}{>h}{G}{>h}}$
in $F$ and $G$ (see (\ref{eq:5})), we obtain also the following
result.

\begin{lemma}\label{lem:h3.plus}
  Let $h$ be any positive function on $\Rp$
  such that $h(x)\to\infty$. Then,
  for any measures $F_1$, $F_2$, $G_1$ and $G_2$ on $\R$,
  $$
    \limsup_{x\to\infty}
    \frac{\ov{\con{(F_1)}{>h}{(G_1)}{>h}}(x)}
    {\ov{\con{(F_2)}{>h}{(G_2)}{>h}}(x)}
    \le \limsup_{x\to\infty}\frac{\ov{F_1}(x)}{\ov{F_2}(x)}
    \cdot \limsup_{x\to\infty}\frac{\ov{G_1}(x)}{\ov{G_2}(x)}.
  $$
  In particular, in the case where the limits
  of the ratios $\ov{F_1}(x)/\ov{F_2}(x)$
  and $\ov{G_1}(x)/\ov{G_2}(x)$ exist, we have
  \begin{displaymath}
    \lim_{x\to\infty}
    \frac{\ov{\con{(F_1)}{>h}{(G_1)}{>h}}(x)}
    {\ov{\con{(F_2)}{>h}{(G_2)}{>h}}(x)}
    =  \lim_{x\to\infty}\frac{\ov{F_1}(x)}{\ov{F_2}(x)}
    \cdot \lim_{x\to\infty}\frac{\ov{G_1}(x)}{\ov{G_2}(x)}.
  \end{displaymath}
\end{lemma}

\section{Convolutions of long-tailed distributions}
\label{sec:results-conv-long}

For the remainder of this paper we specialise to \emph{distributions},
i.e.\ to probability measures on $\R$ each of total mass one.  This
simplifies the statements and proofs of our results, enabling us to
dispense with the constants $\ms{F}$, etc, and is in line with most
applications.  Of course our results may nevertheless be translated to
the case of more general finite measures by renormalising the latter,
applying the results, and re-expressing the conclusions in terms of
the original measures.

We denote by $\LL$ the class of long-tailed distributions on $\R$.
We shall say that two distributions $F_1$, $F_2$ on $\R$ are
\emph{tail equivalent} if and only if $\ov{F_1}(x)\sim\ov{F_2}(x)$ as
$x\to\infty$.  The following two theorems, which provide conditions under
which a random shifting preserves tail equivalence, turns out to be of
key importance in studying convolutions where at least one of the
distributions involved belongs to the class~$\LL$.

\begin{theorem}\label{long.add.5}
  Suppose that $F_1$, $F_2$, and $G$ are distributions on $\R$ such
  that $F_1$ and $F_2$ are tail equivalent.
  If $G\in\LL$ then the distributions
  $F_1*G$ and $F_2*G$ are tail equivalent.
\end{theorem}

\begin{proof}
  Let the function $h$ on $\Rp$ be such that $G$ is
  $h$-insensitive.
  We use the decomposition (\ref{eq:03}).
  It follows from Lemma \ref{lem:h3} that
  \begin{align*}
    \ov{\con{(F_2)}{>h}{G}{}}(x)
    & \sim \ov{\con{(F_1)}{>h}{G}{}}(x).
  \end{align*}
  Further, by Lemma~\ref{lem:h1},
$$
    \ov{\con{(F_2)}{\le h}{G}{}}(x)
    \sim \ov G(x)
    \sim \ov{\con{(F_1)}{\le h}{G}{}}(x),
$$
  so that the conclusion of the theorem follows
  now from (\ref{eq:03}).
\end{proof}

The next theorem generalises a result of Cline \cite{Cline87}
where the case $F_1$, $F_2$, $G_1$, $G_2\in\LL$
was considered.

\begin{theorem}\label{long.add.5.plus}
  Suppose that $F_1$, $F_2$, $G_1$ and $G_2$ are distributions
  on $\R$ such that $\ov F_1(x)\sim\ov F_2(x)$ and
  $\ov G_1(x)\sim\ov G_2(x)$ as $x\to\infty$.
  If the measure $F_1+G_1$ is long-tailed,
  then the distributions $F_1*G_1$ and $F_2*G_2$
  are tail equivalent.
\end{theorem}

\begin{proof}
  The conditions imply that $F_2+G_2$ is
  long-tailed.  Let the function $h$ on $\Rp$ be such that
  $h(x)\le x/2$ and both
  $F_1+G_1$ and $F_2+G_2$ are $h$-insensitive.
  We use the decomposition~(\ref{eq:3}).
  By Lemma \ref{lem:h3.plus},
  \begin{align}
    \ov{\con{(F_2)}{>h}{(G_2)}{>h}}(x)
    & \sim \ov{\con{(F_1)}{>h}{(G_1)}{>h}}(x).
    \label{eq:10.1}
  \end{align}
  Further, by Lemma~\ref{lem:h2},
  \begin{align*}
    \ov{\con{(F_1)}{\le h}{G_1}{}}(x)
    + \ov{\con{F_1}{}{(G_1)}{\le h}}(x)
    & = \ov G_1(x) + \ov{F_1}(x)
    + o(\ov G_1(x) + \ov{F_1}(x))\nonumber\\
    & = \ov G_2(x) + \ov{F_2}(x)
    + o(\ov G_1(x) + \ov{F_1}(x))\nonumber\\
    & = \ov{\con{(F_2)}{\le h}{G_2}{}}(x)
    +\ov{\con{F_2}{}{(G_2)}{\le h}}(x)
    +o(\ov G_1(x) + \ov{F_1}(x)).
  \end{align*}
  It further follows from Lemma~\ref{lem:h2} and from~(\ref{eq:3}) that
  $\ov{F_1}(x)=O(\ov{F_1*G_1}(x))$ and
  $\ov G_1(x)=O(\ov{F_1*G_1}(x))$.
  Thus we obtain
  \begin{equation}\label{eq:12.}
    \ov{\con{(F_1)}{\le h}{G_1}{}}(x)
    +\ov{\con{F_1}{}{(G_1)}{\le h}}(x)
    = \ov{\con{(F_2)}{\le h}{G_2}{}}(x)
    +\ov{\con{F_2}{}{(G_2)}{\le h}}(x)
    + o(\ov{F_1*G_1}(x)).
  \end{equation}
Now the conclusion follows from (\ref{eq:3}),
(\ref{eq:10.1}) and (\ref{eq:12.}).
\end{proof}

Theorems~\ref{long.add.5} and \ref{long.add.5.plus}
have a number of important corollaries, the
first of which is well-known from Embrechts and Goldie \cite{EG80},
but of which we may now give a very simple proof.

\begin{corollary}\label{cor:l2}
  The class~$\LL$ of long-tailed distributions
  is closed under convolutions.
\end{corollary}

\begin{proof}
  Suppose that $F$, $G\in\LL$.  Fix $y>0$.  Define the distribution $F_y$
  to be equal to $F$ shifted by $-y$, that is,
  $\ov{F_y}(x)=\ov F(x+y)$.  Then $F_y*G$ is equal to $F*G$ shifted
  by $-y$.  Since $F\in\LL$ it follows that
  $F$ and $F_y$ are tail-equivalent,
  and since also $G\in\LL$ it follows
  from Theorem~\ref{long.add.5} that
  $F*G$ and $F_y*G$ are tail-equivalent.
  Hence $\ov{F*G}(x)\sim\ov{F*G}(x+y)$,
  implying that $F*G\in\LL$.
\end{proof}

We also have the following, and new, generalisation of
Corollary~\ref{cor:l2}, the proof of which is identical, except that
we must appeal to the (slightly less straightforward)
Theorem~\ref{long.add.5.plus} in place of Theorem~\ref{long.add.5}.

\begin{corollary}\label{cor:l1}
  Suppose that the distributions $F$ and $G$ are such that $F\in\LL$
  and $F+G$ is long-tailed.  Then the distribution $F*G\in\LL$.
\end{corollary}

We remark also that a further special case of Corollary~\ref{cor:l1}
arises when we have $F\in\LL$ and $G$ is such that
$\ov G(x)=o(\ov F(x))$ as $x\to\infty$, so that here again
$F*G\in\LL$.

We give two further general, and known, theorems for convolutions of
long-tailed distributions.

\begin{theorem}\label{F.G..long.inequa}
  Suppose that the distributions $F_1$, \dots,
  $F_n\in\LL$. Then
  \begin{eqnarray*}
    \ov{F_1*\ldots*F_n}(x)
    & \ge & (1+o(1))\sum_{k=1}^n\ov{F_k}(x)
    \qquad\mbox{ as }x\to\infty.
  \end{eqnarray*}
\end{theorem}
\begin{proof}
  It is sufficient to prove the result for the case $n=2$, the general
  result following by induction (since, from Corollary~\ref{cor:l1}
  above, the class~$\LL$ is closed under convolutions).  Let the
  function $h$ be such that $h(x)\le x/2$
  and both $F_1$ and $F_2$ are $h$-insensitive.
  The required result is now immediate from the inequality
  \begin{displaymath}
    \ov{F_1*F_2}(x)
    \ge
    \ov{\con{(F_1)}{\le h}{F_2}{}}(x)+\ov{\con{F_1}{}{(F_2)}{\le h}}(x)
  \end{displaymath}
  and Lemma~\ref{lem:h1} above.
\end{proof}

In particular we have the following corollary (where, as usual,
$F^{*n}$ denotes the $n$-fold convolution of $F$ with itself).

\begin{corollary}\label{co.F.G..long.inequa}
  Suppose that $F\in\LL$.
  Then, for any $n\ge 2$,
  \begin{eqnarray*}
    \liminf_{x\to\infty}\frac{\ov{F^{*n}}(x)}{\ov F(x)}
    &\ge& n.
  \end{eqnarray*}
\end{corollary}

\begin{theorem}\label{long.add.1}
  Suppose that $F\in\LL$.
  Then, for any distribution $G$ on $\R$,
  \begin{eqnarray*}
    \liminf_{x\to\infty}\frac{\ov{F*G}(x)}{\ov F(x)}
    &\ge & 1.
  \end{eqnarray*}
  If, furthermore, the function $h$ on $\Rp$ is such that $F\in\LL$ is
  $h$-insensitive and $\ov G(h(x)) = o(\ov F(x))$ as $x\to\infty$,
  then $\ov{F*G}(x) \sim \ov F(x)$ as $x\to\infty$.  In particular
  this conclusion holds for $F\in\LL$ and any
  distribution $G$ such that $\ov G(a)=0$ for some $a$.
\end{theorem}
\begin{proof}
  Let the function $h$ be such that $F$ is $h$-insensitive.
  We use the decomposition~(\ref{eq:03}) with $F$ and $G$ interchanged.
  From Lemma~\ref{lem:h1},
  \begin{equation}\label{eq:6}
        \ov{\con{F}{}{G}{\le h}}(x)\sim \ov F(x),
  \end{equation}
  and so the first result is immediate.  For the second result note
  that, under the given additional condition and for $h$ as above,
  \begin{displaymath}
    \ov{\con{F}{}{G}{>h}}(x) \le \ov G(h(x))
    = o(\ov F(x)),
  \end{displaymath}
  so that the required result again follows on using
  (\ref{eq:03}) and (\ref{eq:6}).
\end{proof}

\section{Convolutions of subexponential distributions on $\R$}
\label{sec:results-conv-subexp}

We recall that a distribution $F$ on the positive real line~$\Rp$ is
\emph{subexponential} if and only if $\ov F(x)>0$ for all $x$ and
\begin{equation}
  \label{eq:13}
  \ov{F*F}(x) = 2\ov F(x) + o(\ov F(x))
  \quad\mbox{ as }x\to\infty;
\end{equation}
this notion goes back to Chistyakov \cite{Ch}.
By way of interpretation, suppose that, on some probability space with
probability measure~$\Pr$, $\xi_1$ and $\xi_2$ are independent random
variables with common distribution~$F$.  Then, since
$\Pr(\max(\xi_1,\xi_2)>x)=2\ov{F}(x)+o(\ov{F}(x))$, it follows that
the subexponentiality of $F$ is equivalent to the condition that
$\Pr(\xi_1+\xi_2>x)\sim\Pr(\max(\xi_1,\xi_2)>x)$ as $x\to\infty$,
i.e.\ that, for large $x$, the only significant way in which
$\xi_1+\xi_2$ can exceed $x$ is that either $\xi_1$ or $\xi_2$ should
itself exceed $x$.  This is the well-known ``principle of a single big
jump'' for sums of subexponentially distributed random variables.

It is also well-known \cite{Ch} (see also \cite{AN}) that if $F$ on $\Rp$ is subexponential
then $F\in\LL$.  (To see this, note that, for any distribution $F$ on
$\Rp$ such that $\ov F(x)>0$ for all $x$, for any $a>0$ and for all
$x\ge{}a$,
\begin{align*}
  \ov{F*F}(x)
  & = \ov{\con{F}{}{F}{[0,x-a]}}(x)
  + \ov{\con{F}{}{F}{(x-a,x]}}(x)
  + \ov{\con{F}{}{F}{(x,\infty)}}(x)\\
  & \ge \ov F(x)F(x-a) + \ov F(a)F(x-a,x] +
  \ov F(x);
\end{align*}
since $F(x-a)\to1$ and $\ov F(a)>0$,
it follows that under the condition~(\ref{eq:13}) we have
$F(x-a,x]=o(\ov F(x))$ as $x\to\infty$, and so $F\in\LL$ as
required.)

Following recent practice, we extend the definition of
subexponentiality to distributions on the entire real line by saying
that a distribution $F$ on $\R$ is subexponential if and only if
$F\in\LL$ and (\ref{eq:13}) holds (the latter condition no longer
being sufficient to ensure $F\in\LL$). We write $\SSR$ for the class
of subexponential distributions on $\R$.

We shall make repeated use of the following observation
which follows from the upper estimate~(\ref{eq:3.upper}),
Lemma~\ref{lem:h1} (or Lemma~\ref{lem:h2}) and (\ref{eq:13}):
\begin{lemma}[Asmussen et al. \cite{AFK}]\label{eq:14}
If $F\in\LL$ then the following are equivalent:

{\rm(i)} $F\in\SSR$;

{\rm(ii)} for every function~$h$ such that $h(x)\to\infty$,
$\ov{\con{F}{>h}{F}{>h}}(x) = o(\ov F(x))$ as $x\to\infty$;

{\rm(iii)} for some function~$h$ such that $h(x)\to\infty$ and $F$
is $h$-insensitive,
$\ov{\con{F}{>h}{F}{>h}}(x) = o(\ov F(x))$ as $x\to\infty$.
\end{lemma}

The following statement provides the foundation for our results on
convolutions of subexponential distributions.
\begin{lemma}
  \label{lem:s1}
  Suppose that $F\in\SSR$ and that the function~$h$ is such that $h(x)\to\infty$
  as $x\to\infty$. Let distributions $G_1$, $G_2$ be such
  that, for $i=1$, $2$, we have $\ov{G_i}(x)=O(\ov{F}(x))$ as
  $x\to\infty$.  Then
  $\ov{\con{(G_1)}{>h}{(G_2)}{>h}}(x)=o(\ov F(x))$ as $x\to\infty$.
\end{lemma}

\begin{proof}
  Since $\ov G_i(x)=O(\ov F(x))$,
  it follows from Lemma~\ref{lem:h3.plus} that
$$
\limsup_{x\to\infty}
\frac{\ov{\con{(G_1)}{>h}{(G_2)}{>h}}(x)}
{\ov{\con{F}{>h}{F}{>h}}(x)}
\le \limsup_{x\to\infty}
\frac{\ov{G_1}(x)}{\ov F(x)}
\cdot \limsup_{x\to\infty}
\frac{\ov{G_2}(x)}{\ov F(x)}<\infty.
$$
Hence, it follows from Lemma~\ref{eq:14}
since $F\in\SSR$ that
$$
\ov{\con{(G_1)}{>h}{(G_2)}{>h}}(x)
= O(\ov{\con{F}{>h}{F}{>h}}(x))
= o(\ov F(x)).
$$
Hence we have the required result.
\end{proof}

We shall say that distributions $F$ and $G$ on $\R$ are \emph{weakly tail
  equivalent} if both $\ov F(x)=O(\ov G(x))$ and
$\ov{G}(x)=O(\ov{F}(x))$ as $x\to\infty$.  We now have the following
corollary to Lemma~\ref{lem:s1}.
\begin{corollary}[Kl\"uppelberg \cite{K}]
  \label{cor:s0}
  Suppose that $F\in\SSR$, that $G\in\LL$, and that $F$ and $G$ are weakly
  tail-equivalent.  Then $G\in\SSR$.
\end{corollary}
\begin{proof}
  Choose the function $h$ on $\Rp$ so that both $F$ and $G$ are
  $h$-insensitive.  Then, from Lemma~\ref{lem:s1} and the given weak
  tail-equivalence, $\ov{\con{G}{>h}{G}{>h}}(x)=o(\ov G(x))$, and so
  it follows from Lemma~\ref{eq:14} that $G\in\SSR$.
\end{proof}

\begin{remark}
  It follows in particular, as is well known, that subexponentiality
  is a \emph{tail property} of a distribution, i.e.\ for any given
  $x_0$ it depends only on the restriction of the distribution to the
  right of $x_0$.  (Indeed this result also follows from
  Theorem~\ref{long.add.5.plus}: suppose that distributions
  $F_1$, $F_2$ on $\R$ are tail equivalent and that $F_1\in\SSR$; since
  $F_1\in\LL$ we have also $F_2\in\LL$; hence, from
  Theorem~\ref{long.add.5.plus}, on identifying $F_i$ with $G_i$ for
  $i=1,2$ and on using \eqref{eq:13}, we have $F_2\in\SSR$.)  Thus
  also a distribution $F$ on $\R$ is subexponential if and only if the
  distribution $F^+$ on $\Rp$, given by $\ov{F^+}(x)=\ov F(x)$ for
  $x\ge0$ and $F^+(x)=0$ for $x<0$, is subexponential; this provides
  an alternative definition of subexponentiality on $\R$.
\end{remark}

We now have the following theorem.

\begin{theorem}\label{thm:s1}
  Let (a reference distribution) $F\in\SSR$.  Suppose that
  distributions $G_1$, \dots, $G_n$ are such that, for each $k$, the
  measure $F+G_k$ is long-tailed and $\ov{G_k}(x)=O(\ov F(x))$ as
  $x\to\infty$. Then
  \begin{equation}\label{eq:15}
    \ov{G_1*\dots*G_n}(x)
    = \sum_{i=1}^n\ov{G_i}(x)
    + o(\ov F(x))
    \quad\mbox{ as }x\to\infty.
  \end{equation}
\end{theorem}

\begin{corollary}\label{cor:s1}
  Let $F\in\SSR$ and so $F\in\LL$.
  Suppose that distributions $G_1$, \dots, $G_n$
  are such that, individually for each $k$, either
  {\rm(i)} $G_k\in\LL$ and
  $\ov{G_k}(x)=O(\ov F(x))$ as $x\to\infty$ or
  {\rm(ii)}~$\ov G_k(x)=o(\ov F(x))$ as $x\to\infty$.
Then {\rm(\ref{eq:15})} holds.
\end{corollary}

The latter corollary was proved in \cite{EGV}
for the case $n=2$, $G_1=F$, $\ov G_2(x)=o(\ov F(x))$.

\begin{proof}[Proof of Theorem~\ref{thm:s1}]
  Note first that it follows from the conditions of the theorem that,
  for each $i$ and for any constant~$a$,
  \begin{align*}
    \ov F(x+a) + \ov{G_i}(x+a)
    & = \ov F(x) + \ov{G_i}(x) + o(\ov F(x) + \ov{G_i}(x)) \\
    & = \ov F(x) + \ov{G_i}(x) + o(\ov F(x)).
  \end{align*}
  Hence from the representation
  $F+\sum_{i=1}^kG_i=\sum_{i=1}^k(F+G_i)-(k-1)F$ and since $F$ is also
  long-tailed, for each $k$ and for any constant~$a$,
  \begin{displaymath}
     \ov F(x+a) + \sum_{i=1}^k\ov{G_i}(x+a)
     = \ov F(x) + \sum_{i=1}^k\ov{G_i}(x) + o(\ov F(x)),
  \end{displaymath}
  and so the measure $F+\sum_{i=1}^kG_i$ is also long-tailed.  Note
  also that for each $k$ we have $\sum_{i=1}^k\ov{G_i}(x)=O(\ov
  F(x))$.  It now follows that it is sufficient to prove the theorem
  for case $n=2$, the general result then following by induction.

  Let the function $h$ on $\Rp$ be such that $h(x)\le x/2$
  and all $F$, $F+G_1$ and
  $F+G_2$ are $h$-insensitive.  It then follows from
  Lemma~\ref{lem:h1} that, as $x\to\infty$,
  \begin{align}
    \ov{\con{G_1}{}{(G_2)}{\le h}}(x)
    & = \ov{\con{(G_1+F)}{}{(G_2)}{\le h}}(x)
    -\ov{\con{F}{}{(G_2)}{\le h}}(x)\nonumber\\
    & = \ov{G_1+F}(x)
    - \ov{F}(x) +
    o(\ov{G_1}(x)+\ov F(x))\nonumber\\
    & = \ov{G_1}(x)
    + o(\ov F(x)),\label{eq:16}
  \end{align}
  and similarly
  \begin{align}
    \ov{\con{(G_1)}{\le h}{G_2}{}}(x)
    & = \ov{G_2}(x) + o(\ov F(x)).\label{eq:17}
  \end{align}
  Further, from Lemma~\ref{lem:s1},
  \begin{equation}\label{eq:19}
    \ov{\con{(G_1)}{>h}{(G_2)}{>h}}(x) = o(\ov F(x)).
  \end{equation}
  The required result~(\ref{eq:15}) now follows from the
  decomposition~\eqref{eq:3} and from \eqref{eq:16}--\eqref{eq:19}.
\end{proof}

The following result strengthens the conditions
of Theorem~\ref{thm:s1} to provide a sufficient
condition for the convolution obtained there
to be subexponential.

\begin{theorem}\label{thm:s2}
  Suppose again that the conditions of
  Theorem~{\rm\ref{thm:s1}} hold, and
  that additionally $G_1$ satisfies the stronger condition that
  $G_1\in\LL$ and that $G_1$ is weakly tail equivalent to $F$.
  Then $G_1*\dots*G_n \in \SSR$,
  and additionally $G_1*\dots*G_n$ is weakly tail equivalent to $F$.
\end{theorem}

\begin{proof}
  It follows from Corollary~\ref{cor:s0} that $G_1\in\SSR$.  Further
  the  weak tail equivalence of $F$ and $G_1$
  implies that, for each $k$, $\ov{G_k}(x)=O(\ov{G_1}(x))$.
  Hence by Theorem \ref{thm:s1} with $F=G_1$, the distribution
  $G_1*G_2*\dots*G_n$ is long-tailed and weakly
  tail equivalent to $G_1$ and so also to $F$.
  In particular, again by Corollary~\ref{cor:s0},
  $G_1*\dots*G_n\in\SSR$.
\end{proof}


We have the following two corollaries of
Theorems~\ref{thm:s1} and \ref{thm:s2}.
The first is new (the version with $G\in\LL$ was proved in
Embrechts and Goldie \cite{EG80}), while the  other is well-known
(and goes back to \cite{EG82} where the case $n=2$,
$G_1=G_2$ was considered; some particular results may be
found in Teugels \cite{T} and Pakes \cite{P}, see also \cite{AFK}).

\begin{corollary}\label{cor:14}
  Suppose that distributions $F$ and $G$ are such that $F\in\SSR$,
  that $F+G$ is long-tailed and that $\ov G(x)=O(\ov F(x))$ as
  $x\to\infty$.  Then $F*G\in\SSR$ and
  \begin{displaymath}
    \ov{F*G}(x) = \ov F(x) + \ov G(x) + o(\ov F(x))
    \quad\mbox{ as }x\to\infty.
  \end{displaymath}
\end{corollary}

\begin{proof}
  The result follows from Theorems~\ref{thm:s1} and \ref{thm:s2} in
  the case $n=2$ with $G_1$ replaced by $F$ and $G_2$ by $G$.
\end{proof}


\begin{corollary}\label{cor:15}
  Suppose that $F\in\SSR$.  Let $G_1,\dots,G_n$ be distributions
  such that $\ov{G_i}(x)/\ov F(x)\to c_i$ as $x\to\infty$,
  for some constants $c_i\ge0$, $i=1, \dots, n$. Then
  \begin{displaymath}
    \frac{\ov{G_1*\ldots*G_n}(x)}{\ov F(x)}
    \to \sum_{i=1}^n c_i
    \quad\mbox{ as }x\to\infty.
  \end{displaymath}
  If $c_1+\ldots+c_n>0$, then
  $G_1*\ldots*G_n\in\SSR$.
\end{corollary}

\begin{proof}
  The first statement of the corollary is immediate from
  Theorem~\ref{thm:s1}.  If $c_1+\ldots+c_n>0$, we may assume without
  loss of generality that $c_1>0$, so that the second statement
  follows from Theorem~\ref{thm:s2}.
\end{proof}

\section{Closure properties of subexponential distributions}
\label{sec:clos-prop-subexp}

It is well-known that the class of \emph{regularly varying}
distributions, which is a subclass of the class $\SSR$ of
subexponential distributions, is closed under convolution.\footnote{
A distribution $F$ is regularly varying if its right tail $\ov F$ is
a regularly varying function. This means that the latter can
be represented as $\ov F(x) = x^{-\alpha} l(x)$ for all $x$ positive
where $l(x)$ is a {\it slowly varying} function, i.e. $l$ is strictly
positive and $l(cx) \sim l(x)$ as $x\to\infty$, for any positive
constant $c$.
}  Indeed if
$F$ and $G$ are regularly varying, the result that $F*G$ is also
regularly varying is straightforwardly obtained from
Theorem~\ref{thm:s1} by taking the ``reference'' distribution of that
theorem to be $(F+G)/2$.  It is also known that the class $\SSR$
does not possess this closure property.  However, if distributions
$F$, $G\in\SSR$, then it follows from Corollary~\ref{cor:14} that a
sufficient condition for $F*G\in\SSR$ is given by
$\ov G(x)=O(\ov F(x))$ as $x\to\infty$.  (Indeed, as the corollary
shows, $G$ may satisfy weaker conditions than that of being
subexponential.)  Further it follows that under this condition we have
that, for any function $h$ such that both $F$ and $G$ are
$h$-insensitive,
\begin{equation}\label{eq:21}
  \ov{\con{F}{>h}{G}{>h}}(x) = o(\ov F(x)+\ov G(x))
  \quad\mbox{ as }x\to\infty.
\end{equation}
(See, for example, the proof of Theorem~\ref{thm:s1} above.)  The
following result is therefore not surprising:
if $F$, $G\in\SSR$, the condition~\eqref{eq:21}
is \emph{necessary and sufficient} for
$F*G\in\SSR$.  This is one of the results given by
Theorem~\ref{closure.of.S} below.  The first three equivalences given
by the theorem are known from Embrechts and Goldie \cite{EG80};
the novelty of the result lies in the fact
that each of them is equivalent to the condition~(iv)
(condition~\eqref{eq:21} above) and in that we give new and very short
proofs of these equivalences.

\begin{theorem}\label{closure.of.S}
  Suppose that the distributions $F$ and $G$ on $\R$ are
  subexponential.  Suppose also that $p$ is any constant such that
  $0<p<1$, and that the function $h$ on $\Rp$
  is such that $h(x)\le x/2$ and both $F$
  and $G$ are $h$-insensitive. Then the following conditions are
  equivalent:

  {\rm(i)} $\ov{F*G}(x) \sim \ov F(x) + \ov G(x)$
  as $x\to\infty$;

  {\rm(ii)} $F*G\in\SSR$;

  {\rm(iii)} the mixture $pF+(1-p)G\in\SSR$;

  {\rm(iv)} $\ov{\con{F}{>h}{G}{>h}}(x) = o(\ov F(x)+\ov G(x))$
  as $x\to\infty$.
\end{theorem}

\begin{proof}
  We show that each of the conditions~(i)--(iii) is
  equivalent to the condition~(iv).
  First, since $F$ and $G$ are subexponential, and hence
  long-tailed, it follows from the decomposition~\eqref{eq:3} and
  Lemma~\ref{lem:h1} that
  \begin{align}
    \ov{F*G}(x)
    & = \ov{\con{F}{}{G}{\le h}}(x)
    + \ov{\con{F}{\le h}{G}{}}(x)
    + \ov{\con{F}{>h}{G}{>h}}(x) \nonumber\\
    & = \ov F(x)+o(\ov F(x))
    + \ov G(x)+o(\ov G(x))
    + \ov{\con{F}{>h}{G}{>h}}(x).\label{eq:22}
  \end{align}
  Hence the conditions~(i) and (iv) are equivalent.

  To show the equivalence of (ii) and (iv) observe first that the
  subexponentiality of $F$ and $G$ implies that
  \begin{equation}
    \label{eq:24}
    \ov{F^{*2}}(x) \sim 2\ov F(x),
    \qquad \ov{G^{*2}}(x) \sim 2\ov G(x),
  \end{equation}
  and thus in particular, from Lemma~\ref{lem:h3.plus}, that
  \begin{equation}
    \label{eq:7}
    \ov{\con{(F^{*2})}{>h}{(G^{*2})}{>h}}(x)
    \sim
    4\ov{\con{F}{>h}{G}{>h}}(x).
  \end{equation}
  Further, since $(F*G)^{*2}=F^{*2}*G^{*2}$ and since both $F^{*2}$
  and $G^{*2}$ are $h$-insensitive, $\ov{(F*G)^{*2}}(x)$ may  be
  estimated as in~\eqref{eq:22} with $F^{*2}$ and $G^{*2}$ replacing
  $F$ and $G$.  Hence, using also (\ref{eq:24}) and (\ref{eq:7}),
  \begin{equation}
    \label{eq:25}
    \ov{(F*G)^{*2}}(x)
    = (2+o(1))(\ov F(x)
    + \ov G(x))
    + (4+o(1))\ov{\con{F}{>h}{G}{>h}}(x).
  \end{equation}
  Now  since the subexponentiality of $F$ and $G$ also implies,
  by Corollary~\ref{cor:l2}, that $F*G\in\LL$, the condition~(ii) is
  equivalent to the requirement that
  \begin{align}
    \ov{(F*G)^{*2}}(x)
    & = (2+o(1))\ov{F*G}(x) \nonumber\\
    & = (2+o(1))(\ov F(x)
    + \ov G(x))
    + (2+o(1))\ov{\con{F}{>h}{G}{>h}}(x),\label{eq:23}
  \end{align}
  where (\ref{eq:23}) follows from (\ref{eq:22}).  However, the
  equalities~(\ref{eq:25}) and (\ref{eq:23}) hold simultaneously if
  and only if $\ov{\con{F}{>h}{G}{>h}}(x)=o(\ov F(x)+\ov G(x))$, i.e.\
  if and only if the condition~(iv) holds.

  Finally, to show the equivalence of (iii) and (iv),
  note first that $pF+(1-p)G$ is $h$-insensitive.
Hence, by Lemma \ref{eq:14}, the subexponentiality of
$pF+(1-p)G$ is equivalent to
$$
\ov{\con{(pF+(1-p)G)}{>h}{(pF+(1-p)G)}{>h}}(x)
    = o(\ov F(x)+\ov G(x)).
$$
The left side is equal to
$$
    p^2\ov{\con{F}{>h}{F}{>h}}(x)
    +(1-p)^2\ov{\con{G}{>h}{G}{>h}}(x)
    +2p(1-p)\ov{\con{F}{>h}{G}{>h}}(x).
$$
By the subexponentiality of $F$ and $G$ and again by
Lemma \ref{eq:14}, $\ov{\con{F}{>h}{F}{>h}}=o(\ov F(x))$
and $\ov{\con{G}{>h}{G}{>h}}=o(\ov G(x))$.
The equivalence of  (iii) and (iv) now follows.
\end{proof}

\section*{Acknowledgment}
\label{sec:acknowledgment}

The authors acknowledge the hospitality of the Mathematisches
Forschungsinstitut Oberwolfach, where much of the research reported
here was carried out.

\end{document}